\documentclass[11pt, a4paper]{article}

\usepackage{inputenc}
\usepackage{graphicx}
\usepackage{amssymb}
\usepackage{mathtools}
\usepackage{amsmath}
\usepackage{comment}
\usepackage{caption}
\usepackage{subcaption}
\usepackage{algorithm}
\usepackage{algorithmic}
\usepackage[colorlinks=true,linkcolor=black]{hyperref}
\usepackage{xcolor}
\usepackage{enumitem}
\usepackage{multicol}
\usepackage{multirow}
\usepackage{booktabs}

\newcommand{\R}{\mathbb R}

\newtheorem{definition}{Definition}
\newtheorem{theorem}{Theorem}

\newtheorem{proof}{Proof}

\begin{document}
	
%\maketitle	

\begin{center}
{\large\bf{
Rational activation functions in neural networks with uniform based loss functions and its application in classification.}  }
\end{center}

\begin{center}
	{Vinesha Peiris}
\end{center}

\begin{abstract}
In this paper, we demonstrate the application of generalised rational uniform (Chebyshev) approximation in neural networks. In particular, our activation functions are one degree rational functions and the loss function is based on the uniform norm. In this setting, when the coefficients of the rational activation function are fixed, the overall optimisation problem of the neural network forms a generalised rational uniform approximation problem where the weights and the bias of the network are the decision variables. To optimise the decision variables, we suggest using two prominent methods: the bisection method and the differential correction algorithm. 

We perform numerical experiments on classification problems with two classes and report the classification accuracy obtained by the network using the bisection method, differential correction algorithm along with the standard MATLAB toolbox which uses least square loss function. We show that the choice of the uniform norm based loss function with rational activation function and the bisection method lead to better classification accuracy when the training dataset is either very small or if the classes are imbalanced. 
\end{abstract}

{\bf Keywords}: 
Quasiconvex optimisation; Generalised rational uniform approximation; Rational activation functions; Deep learning.

{\bf MSC2010}: 
90C90,              % Applications of mathematical programming
90C05, % Linear prgramming
90C47,            %minimax
68T99.             % Artificial intelligence, general
%47H05, 90C33.}

\section{Introduction} \label{sec:intro}

Deep learning is a branch of machine learning which is inspired by artificial neural networks. Recently, the study of deep learning has been an interesting area for many researchers due to its excellent performance (fast and efficient) in a variety of practical applications~\cite{CAO201817,goodfellow2016deep,lecun2015deep,raissi2018deep,min2017deep}. In general, a neural network contains layers (input, hidden, output) of interconnected nodes and these connections are known as weights. Hidden layers refine the weights coming from the previous layer with the help of a nonlinear activation function. From the point of view of optimisation, the main goal of a network is to optimise the weights by minimising the loss function. Hence, the choice of the activation function, the format of the loss function and the learning algorithm which minimises the loss function have a significant impact on the efficiency of the network.

Rational activation functions have drawn much attention lately due to its flexibility and smoothness~\cite{boull2020a,chen2018rational,delfosse2021recurrent,molina2019pade}. Telgarsky's theoretical work~\cite{telgarsky2017neural} on error bounds of the approximation to the ReLU network by rational functions and vice versa influenced several recent rational function adaptations to neural networks as an activation function~\cite{boull2020a,molina2019pade}.
In the most recent work~\cite{boull2020a}, authors study rational neural networks, whose activation functions are low degree rational functions with trainable coefficients. In all prior work on rational activation functions, the loss function is based on the least squares formulation. 

The least squares loss function is often employed for optimising the parameters (weights) of a neural network against a training set. This is due to the fact that basic optimisation techniques such as gradient descent may be effectively utilised to optimise the loss function. In some specific cases, the loss function can be successfully adapted to a uniform (Chebyshev) approximation based model where the maximum error is minimised. In particular, if the size of the data available for training is small but reliable or if the data is greately imbalanced, the uniform approximation based model performs much better than least square based models~\cite{peiris2021}. 

In this paper, we consider a simple neural network model with no hidden layers whose activation function is a classical rational function (ratio of two polynomials, the basis functions are just monomials) of degree (1,1) and the loss function is in the form of uniform approximation. The coefficients of the rational activation function are fixed by the best rational approximation to the ReLU function. In these circumstances, the overall optimisation problem of the network is naturally formulated as a generalised rational uniform approximation problem whose basis functions are the inputs of the neural network. Weights and the bias term are the decision variables to be optimised.

The generalised rational uniform approximations in the sense of Cheney and Loeb~\cite{cheney1964generalized} are well-studied and there are many efficient techniques to construct these approximations~\cite{DiffCorrection1972,cheney1961DCoriginal,loeb1960,trefethen2018,AMCPeirisSukhSharonUgon}. The differential correction method~\cite{DiffCorrection1972,cheney1961DCoriginal,cheney1962DC} is an iterative scheme which has guranteed convergence properties to find the optimal generalised rational approximation.
Bisection method is another simple, robust and modern optimisation based approach which uses the quasiconvexity property of the corresponding optimisation problems~\cite{SL}. The class of quasiconvex functions includes all the functions whose sublevel set is convex. Authors of~\cite{AMCPeirisSukhSharonUgon} use this method to construct univariate generalised rational uniform approximations and later, it was extend in several directions~\cite{millan2021multivariate,PeirisSukhorukova}.
In this paper, we prove that the relevent optimisation problems appearing in a simple neural network with one degree rational activation functions and uniform based loss function is quasiconvex and therefore, it can be solved using the bisection method.

We implement the differential correction algorithm and bisection method in MATLAB adapted to a simple neural network with no hidden layers. We perform our numerical experiments on classification problems with two classes and report the classification accuracy obtained by the network using the bisection method, differential correction algorithm along with the standard MATLAB toolbox which uses least square loss function. We highlight the benefits of the bisection method in connection with neural networks over the differential correction algorithm by comparing the computational time and the classification accuracy.

This paper is organised as follows. Section~\ref{sec:network_model} provides the notations and the model formulation of a simple neural network. In Section~\ref{sec:multi_gen_rat_Cheb}, we define generalised rational uniform approximations and includes definitions of quasiconvex functions. In Section~\ref{sec:our_model}, we introduce one degree rational activation functions and incorporate it into a simple neural network whose loss function is based on uniform approximation. Section~\ref{sec:train_the_model} discusses training of the network using the bisection method and the differential correction algorithm. Section~\ref{sec:numerical_experiments} provides numerical experiments. Finally, Section~\ref{sec:conclusions} explains conclusions and future research directions.

\section{Formulating a general neural network model} \label{sec:network_model}

Consider a neural network which consists of an input layer with $n$ nodes and an output layer with $m$ nodes. Let $x = (x_1,\dots, x_n)\in \R^n$ represents the input of the neural network and $y\in \R^m$ represents the output computed by the neural network. The weights between the input nodes and the output nodes are denoted by $w_{ij}$ where $i=1:m$, $j=1:n$. The forward propagation that pass the input to the next layer are computed by taking the linear combination between inputs and the corresponding weights.

$$z_i (x) = \sum_{j=1}^{n} w_{ij} x_{j} + b_j, \quad i=1:m$$
where $b_j, j=1:n$ is known as the bias term.  
If the activation function defined on the output layer is $\sigma$, then the output of the ntework is a single composition between $\sigma$ and $z_{i}$.

$$\displaystyle y_{i} (x) = \sigma(z_i (x)) = \sigma \left ( \sum_{j=1}^{n} w_{ij} x_{j} + b_j \right ), \quad i=1:m.$$

This can also be formulated in matrix notation as follows:
$$y(x) = \sigma({\bf W}x+b),$$
where $${\bf W} = 
\begin{bmatrix}
	w_{11} & w_{12} & \ldots & w_{1n}\\
	\vdots & \vdots & \ddots & \vdots\\
	w_{m1} & w_{m2} & \ldots & w_{mn}
\end{bmatrix}.$$

The training of a neural network is done by minimising the loss function (optimising the weights of the network). The loss function calculates how well the network produced output ($y_i$) describes the expected output values of a given training set. 

Let $\{(x^1,y^1),(x^2,y^2), \ldots, (x^N,y^N)\}$ be a training set where $x^i, i=1:N$ are the inputs and $y^i, i=1:N$ are the desired outputs. Weights and the bias term are the decision variables of the corresponding optimisation problems of the network and one needs to optimise the decision variables ($\bf W$, $b$) such that the error between the neural network output $y(x^i) = \sigma({\bf W}x^i+b),$ and the desired output of the training set $y^i$ has to be minimised in some sense. A common loss function used in neural networks is the least squares loss function,
\begin{equation*}
	L({\bf W},b) = \sum_{i=1}^N (y^i-\sigma({\bf W}x^i+b))^2.
\end{equation*}
Gradient descent algorithms can be accompanied to optimise the decision variables of the least sqaures loss function when training the network. There are many different types of loss functions used in neural networks, and~\cite{goodfellow2016deep} goes over a few of them in detail.

In this paper, we use uniform approximation based formula for the loss function,
\begin{equation*}
	L({\bf W},b) = \max_{\substack{i\in \{1:N\} \\ j\in \{1:m\}}} |y_{j}^i-\sigma_{j}({\bf W}x^i+b)|.
\end{equation*}
This is also known as Chebyshev norm and max norm. This loss function is much more efficient than the mean least squares based approach when the training set is reliable, but limited in size. 

In this research, we only consider a simple neural network with no hidden layers and we assume that the output layer consist of a single node. Then the uniform based loss function is
\begin{equation*}
	L({\bf W},b) = \max_{i=1:N} |y^i-\sigma({\bf W}x^i+b)|,
\end{equation*}
and ${\bf W}$ is now reduced to a row vector.

\section{Generalised rational uniform approximations} \label{sec:multi_gen_rat_Cheb}

In this section, we define the optimisation problems appearing in generalised rational uniform approximations and discuss some important definitions of quasiconvex functions. First, let us define the generalised rational functions in terms of Cheney and Loeb~\cite{cheney1964generalized}. In particular, these generalised rational functions are the ratios of linear forms,
\begin{equation}\label{eq:genrational}
	\bar{R}_{n,m}(x) =  P(x)/Q(x) = \frac{\sum_{i=0}^{n} p_i g_i(x)}{\sum_{j=0}^{m} q_jh_j(x)},
\end{equation}
where $g_i(x)$, $i=1:n$ and $h_j(x)$, $j=1:m$ are basis functions and $x \in [c,d]$, a closed interval on the real line. Functions from~(\ref{eq:genrational}) reduces to the classical rational functions when the basis functions are just monomials, i.e., ratio of two polynomials.

In the case of generalised rational uniform approximation, the optimisation problem is
\begin{equation}\label{eq:problem}
	\min \max_{x \in [c,d]}\left|f(x)-\frac{{\bf A}^T{\bf G}(x)}{{\bf B}^T{\bf H}(x)}\right|
\end{equation}
subject to
\begin{equation}\label{eq:positivity}
	{\bf B}^T{\bf H}(x)>0,~x\in [c,d],
\end{equation}
where 
\begin{itemize}
	\item $f(x)$ is the function to approximate,
	\item ${{\bf A}}=(a_0,a_1,\dots,a_n)^T\in\R^{n+1}, {{\bf B}}=(b_0,b_1,\dots,b_m)^T\in\R^{m+1}$ are the decision variables,
	\item ${{\bf G}}(x)=(g_0(x),\dots,g_n(x))^T$, ${{\bf H}}(x)=(h_0(x),\dots,h_m( x))^T$, where $g_j(x)$, $j=1:n$ and $h_i(x)$, $i=1:m$ are known basis functions.
\end{itemize}
 
Therefore, the approximations are ratios of linear combinations of basis functions. When all the basis functions are monomials, the problem is reduced to the best rational uniform approximation.

In~\cite{AMCPeirisSukhSharonUgon}, the authors proved that the objective function of~(\ref{eq:problem}) forms a quasiconvex function. The definition of quasiconvex functions and associated preliminary findings are included in the next section.

\subsection{Quasiconvex functions}  \label{ssec:quasiconvex_functions}

\begin{definition}  \label{def:quasiconvex_2}
	Function $ f(t)$ is quasiconvex if and only if its sublevel set 
	\begin{displaymath}
		S_{\alpha}=\{x |f(x) \leq \alpha \} 
	\end{displaymath}
	%$$ S_{\alpha}=\{x |f(x) \leq \alpha \} $$
	is  convex for any $\alpha \in \mathbb{R}$.  The set $S_{\alpha}$ is also called $\alpha$-sublevel set.
\end{definition}
It can be shown that this Definition~\ref{def:quasiconvex_2} is equivalent to the following definition~\cite{fenchel1953}.

\begin{definition} \label{def:quasiconvex_1}
	A function $ f : D \rightarrow \mathbb{R} $ defined on a convex subset $D $ of a real vector space is called  quasiconvex if and only if  for any pair $x$ and $y$ from $D$  and $\lambda\in [0,1]$ one has 
	\begin{displaymath}
		f (\lambda x + ( 1-\lambda ) y )
		\leq
		\max\{ f ( x ) , f ( y )\}. 
	\end{displaymath} 
\end{definition}

\begin{definition} \label{def:quasiconcave}
	Function $f$ is quasiconcave if and only if $-f$ is quasiconvex.
\end{definition}
This means that every superlevel set $ \bar{S}_{\alpha}=\{x |f (x ) \geq \alpha \} $ is convex.
\begin{definition} \label{def:quasiaffine}
	Functions that are quasiconvex and quasiconcave at the same time are called quasiaffine (sometimes quasilinear).
\end{definition}

Since the problem of generalised rational uniform approximation forms a quasiconvex function, it can be treated using a number of computational methods developed for quasiconvex optimisation~\cite{SL,DaCruzAlgorithmsQuasiconvex,DaniilidisHadjisavvasMartinezLegas2002,dutta2005abstract,MLegazquasiconvexduality,Rubinov00,RubinovSimsek}. One such method (called bisection method for quasiconvex functions~\cite{SL}) will be described in Section~\ref{ssec:bisection}.

The bisection method that we use in this paper is simple and robust, but it is not as efficient as some other methods, in particular, the differential correction method. Differential correction method does not require any special properties of the objective function to process. We discuss this further in Section~\ref{ssec:DC}.

\section{Network with rational activation functions and uniform norm based loss function} \label{sec:our_model}

We start this section by introducing rational activation functions. Activation functions play a cruicial role in the theory of neural networks. The main goal of defining an activation function on the output of a particular layer is to introduce nonlinearity to the network. The choice of the activation function depends on the type of application.
Sigmoid, logistic or hyperbolic tangent are popular smooth activation functions, but their derivatives are zero for large inputs which cause problems when the gradient descent algorithm is employed to optimise the loss function~\cite{bengio1994learning}. An attractive alternative is polynomial activation function~\cite{cheng2018polynomial,goyal2019poly,oh2003polynomial}. At the same time, they are not efficient when one needs to approximate nonsmooth (or non-Lipschitz) functions: high order polynomials are required, leading to severe oscillations and numerical instability. Recently, Rectified Linear Unit (ReLU) and its variants (Leaky ReLU, Exponential Linear Unit, Parametric ReLU, etc.) has been widely used by deep learning community due to its reduced likelihood of vanishing gradient.

\subsection{Rational activation function}  \label{ssec:rational_activation}

Let $R_{n,m}$ be the set of all rational functions of degree $n$ and $m$ where $n$ is the degree of the numerator and $m$ is the degree of the denominator ($n$ and $m$ are nonnegative integers) and $x \in [c,d]$ is a closed segment on the real line. 

\begin{equation} \label{eq:ratio_of_two_polynomials}
	\displaystyle 
	R_{n,m}(x) =  \frac{P(x)}{Q(x)} = \frac{\sum_{i=0}^{n} a_i x^i}{\sum_{j=0}^{m} b_j x^j} = \frac{a_0+a_1x+a_2x^2+ \ldots + a_nx^n}{b_0+b_1x+b_2x^2+ \ldots + b_nx^m}. 
\end{equation}
Functions from $R_{n,m}$ are also known as classical rational functions (ratio of two polynomials) where the denominator is strictly positive.

Telgarsky's theoretical work~\cite{telgarsky2017neural} inspired the use of rational functions in deep learning as an activation function. His work highlights the tight connections between rational functions and neural networks with ReLU activation functions. In~\cite{molina2019pade}, rational activation functions were introduced for the first time as Pad\'e Activation Units (PAU). In the most recent work, low degree rational functions have been employed on rational activation functions~\cite{boull2020a}. Moreover, their numerical experiments suggest that rational neural networks are an attractive alternation to ReLU networks. 

In both~\cite{boull2020a,molina2019pade}, $a_i, i=0:n$ and $b_j,j=1:m$ are learnable parameters of the network, meaning they are not fixed and they are optimised with the rest of the decision variables (weights and bias) of the optimisation problems of the network where the loss function is in the form of least squares. However, the optimisation process of~\cite{boull2020a} is initialised by the coefficients of the best rational approximation to ReLU function while~\cite{molina2019pade} uses the LReLU function. The idea behind this is to initialise the rational networks near a network with ReLU or LReLU activation functions~\cite{boull2020a} since rational functions effectively approximates neural networks with ReLU activations~\cite{telgarsky2017neural}. \cite{boull2020a,molina2019pade} also use slightly different normalisation conditions which are imposed to gurantee that the rational function is irreducible. In particular, authors of~\cite{molina2019pade} fixes the coefficient of the lowest degree monomial of the denominator at $1$ and authors of~\cite{boull2020a} uses $max|b_j| = 1, j=1:m$.

\subsection{Network with one degree rational activation function} \label{ssec:model_one_degree_rat}
%\subsection{Multivariate generalised rational Chebyshev approximation}
Our neural network model is different from the above discussed ones from few different perspectives.
\begin{enumerate}	
	\item Our loss function is based on uniform approximation.
\end{enumerate}
Because of the nature of the uniform norm based loss function, counting the contribution of outliers in the dataset is possible, but least squares formulation discounts them by averaging. Therefore, our loss functions gurantee better classification accuracy if the training data is either limited in size or imbalanced~\cite{peiris2021}. 

This type of limited training datasets are available if the labeling of the dataset is done manually which can be both costly and time consuming in general. Sometimes, if the data are produced by a very expensive procedure, then the available number of data is limited in size~\cite{miles2014}. 
Moreover, in many robotics applications, it is not enough to have accurate performance on average; one requires the maximum variation to be within a specific range~\cite{meltser1996approximating}. There are many research devoted to classification accuracy with limited data or imbalanced class data~\cite{mazurowski2008training,bataineh2017}, however, most studies are based on the popular least squares loss function.

\begin{enumerate}[resume]	
	\item We only utilise one degree rational functions as activation functions and the coefficients of the rational activation are fixed by the best uniform rational approximation to the ReLU function between $[-1,1]$.
\end{enumerate}	
In particular, our activation function is
\begin{equation} \label{eq:our_act}
	R(x) = \frac{a_0 + a_1 x}{b_0 + b_1 x},
\end{equation}
where $b_0 + b_1 x > 0$. This function is smooth, simple and easy to work with in the uniform approximation based loss function. Most importantly, the objective function of the corresponding optimisation problem forms a quasiconvex function when the coefficients of the activation function are fixed.

We formulate our loss function as follows:
\begin{equation*} 
	L({\bf W},b) = \max_{i\in \{1:N\}} \left | y^i- R({\bf W}x^i+b) \right | ,
\end{equation*}
\begin{equation} \label{eq:our_loss}
	L({\bf W},b) = \max_{i\in \{1:N\}} \left | y^i- \frac{a_0+a_1 ({\bf W}x^i +b)}{b_0+b_1 ({\bf W}x^i +b)} \right |,
\end{equation}
where $b_0 + b_1 ({\bf W}x^i +b) > 0$ for all $i=1:N$. 

\begin{theorem}\label{thm:quasiconvexity}
	Function
	\begin{equation}\label{eq:phirational}
		L({\bf W},b) = \max_{i\in \{1:N\}} \left | y^i- \frac{a_0+a_1 ({\bf W}x^i +b)}{b_0+b_1 ({\bf W}x^i +b)} \right |
	\end{equation}
	is quasiconvex.
\end{theorem}

\begin{proof}
	Let $$R({\bf W}x^i+b) = \frac{a_0+a_1 ({\bf W}x^i +b)}{b_0+b_1 ({\bf W}x^i +b)}, \quad i=1:N$$. 
	
	When $a_0, a_1$ and $b_0, b_1$ are fixed, for any known input value $x^i$, the numerator and the denominator of $R$ consist of linear combinations between known basis functions and the parameters ${\bf W},b$. Hence, $R$ is a ratio of two linear functions (ratio of linear forms).
	
	It is shown in~\cite{SL} that ratios of linear forms are quasilinear. Therefore, $R({\bf W}x^i+b)$ is quasilinear. 
	
	Since $|w| = max\{w,-w\}$, one can see that 
	$$L({\bf W},b) = \max_{i\in \{1:N\}} \max \{y^i-R({\bf W}x^i+b), -y^i+R({\bf W}x^i+b)\}$$ is the maximum of quasilinear functions. Quasilinear functions are quasiconvex. Maximum of quasiconvex functions is quasiconvex and hence, $L({\bf W},b)$ is quasiconvex.
\end{proof}

When $a_0, a_1$ and $b_0, b_1$ are fixed, for any pair of training input and desired output $(x^i,y^i)$, the loss function~(\ref{eq:our_loss}) falls under the category of  generalised rational uniform approximation. Various techniques can be used to optimise the decision variables of such approximation problems~\cite{DiffCorrection1972,cheney1961DCoriginal,AMCPeirisSukhSharonUgon}.

\begin{enumerate}[resume]
	\item We use two different algorithms to optimise the parameters of the loss function with two different normalisation conditions. Namely, the bisection method and the differential correction algorithm.
\end{enumerate}	
 
The decision variables of the optimisation problems appearing in our simple neural network are the weights ${\bf W}$ and the bias term $b$. We select two prominent methods to find the optimal weights of the netwrok. In subroutines of both methods, the undelying optimisation problems can be solved using linear programming techniques which are easy to implement in most languages. Bisection method uses the quasiconvex characteristic of the loss function~(\ref{eq:our_loss}) and to gurantee the irreducibility of~(\ref{eq:our_act}), we use the normalisation condition of $b_0=1$. Differential correction algorithm obtain the optimal solution by finding the descent direction at each step of the process and the usual normalisation condition for this method in practice is $max|b_j| = 1, j=0,1$.

{\bf Remark:} If one needs to use rational activation functions with higher degree (degree $\geq 2$ in the numerator or in the denominator), then the quasiconvex property is immediately lost and the application of bisection method is not possible. Moreover, if the coefficients of~(\ref{eq:our_act}) are cosidered to be learnable paramters of the network, rather than fixing them, then the corresponding optimisation problems are more complex and may need polynomial optimisation techniques to solve them. Since the purpose of this study is to investigate the most basic and simple model of the neural network with rational activation functions and uniform based loss function, we leave these research directions open for future studies.

\section{Training the model}  \label{sec:train_the_model}

In this section, our goal is to develop the implementation of the bisection method and the differential correction method for our setting. The optimisation problem that we need to solve is as follows:

\begin{equation}  \label{eq:loss_for_train}
	\min_{{\bf W},b} \max_{i\in \{1:N\}} \left | y^i- R({\bf W}x^i+b) \right | 
\end{equation}
subject to
\begin{equation} \label{eq:denom_for_train}
	b_0+b_1 ({\bf W}x^i +b) > 0, \quad i=1:N,
\end{equation}
where 
$$R({\bf W}x^i+b) = \frac{P({\bf W}x^i+b)}{Q({\bf W}x^i+b)} = \frac{a_0+a_1 ({\bf W}x^i +b)}{b_0+b_1 ({\bf W}x^i +b)}, \quad i=1:N,$$ 
${\bf W}$ is a vector consist of all the weights, $b$ is the bias term and $(x^i,y^i), i=1:N$ is the training set.

\subsection{Bisection method}  \label{ssec:bisection}

The bisection method is a simple, but efficient approach developed for minimising quasiconvex functions~\cite{SL}. The convergence of this method is linear. Theorem~\ref{thm:quasiconvexity} gurantees the quasiconvexity of the objective function~(\ref{eq:loss_for_train}) and therefore, we use the bisection algorithm to find the decision variables of the problem~(\ref{eq:loss_for_train})--(\ref{eq:denom_for_train}). As the normalisation condition, we use $b_0=1$.

Problem~(\ref{eq:loss_for_train})--(\ref{eq:denom_for_train}) can be also formulated as follows:
\begin{equation}  \label{eq:loss_for_train2}
	\min z
\end{equation}
subject to
\begin{equation} \label{eq:const1}
	y^i- R({\bf W}x^i+b) \leq z, \quad i=1:n
\end{equation}
\begin{equation} \label{eq:const2}
	R({\bf W}x^i+b) - y^i \leq z,  \quad i=1:n
\end{equation}
\begin{equation} \label{eq:const3}
	Q ({\bf W}x^i +b) > 0, \quad i=1:n
\end{equation}
Without loss of generality, the constraint~(\ref{eq:const3}) can be replaced by
\begin{equation*} \label{eq:const4}
	1+b_1 ({\bf W}x^i +b) = Q({\bf W}x^i+b) \geq \delta, \quad i=1:n,
\end{equation*}
where $\delta$ is an arbitrary positive number. To initialise the bisection method, one needs to define the following parameters:
\begin{itemize}
	\item The lower bound $l$.
	
	Since the objective function is nonnegative, one can choose $l =0$ as the lower bound for the optimal solution.
	
	\item The upper bound $u$.
	
	One can substitute any value of the decision variables ${\bf W}$ and $b$ in the objective function, for instance, ${\bf W} =0$ and $b=0$, then
	$$u = \max_{i\in \{1:N\}} \left | y^i- R(0) \right |. $$
	$R(0) = a_0$ due to the normalisation condition and therefore, the upper bound is
	$$u = \max_{i\in \{1:N\}} \left | y^i- a_0 \right |. $$

	\item The absolute precision for maximal deviation $\varepsilon$.
\end{itemize}

Set $z = \frac{u+l}{2}$ and check if the set of constraints~(\ref{eq:const1})--(\ref{eq:const3}) has a feasible solution. If this set is feasible, update the upper bound $u=z$, otherwise update the lower bound $l=z$. This procedure needs to be repeated until $u-l \leq \varepsilon$.

In general, checking the feasibility of a set may be difficult. However, in this case, the feasibility problem is reduced to solving a linear programming problem. When $z$ is fixed, one needs to solve the following problem to check the feasibility:
\begin{equation}  \label{eq:loss_for_train3}
	\min u
\end{equation}
subject to
\begin{equation} \label{eq:const5}
	y^i Q({\bf W}x^i+b)- P({\bf W}x^i+b) \leq z Q({\bf W}x^i+b) +u, \quad i=1:N
\end{equation}
\begin{equation} \label{eq:const6}
	P({\bf W}x^i+b) - y^i Q({\bf W}x^i+b) \leq z ({\bf W}x^i+b) +u,  \quad i=1:N
\end{equation}
\begin{equation} \label{eq:const7}
	Q({\bf W}x^i +b) \geq \delta, \quad i=1:N.
\end{equation}
If the optimal solution $u \leq 0$, then the set~(\ref{eq:const1})--(\ref{eq:const3}) has a feasible point, otherwise the set is empty. This feasibility problem~(\ref{eq:loss_for_train3})--(\ref{eq:const7}) is a linear programming problem with respect to its decision variables ${\bf W}, b$ and hence, it can be solved by using any standard linear programming technique.

\subsection{Differential correction method}  \label{ssec:DC}
There are a few different versions of the differential correction algorithm. The original differential correction method was first developed by Cheney and Loeb~\cite{cheney1961DCoriginal}. Later on, a slightly modified versions of the same original method were published in~\cite{cheney1962DC,cheney1963survey,rice1969}. Later on, in 1896, it was discovered that the differential correction method is similar to the Dinkelbach's method when applied to rational approximation problems~\cite{crouzeix1986note}. 

In general, differential correction method has sure convergence properties; it was proved in~\cite{DiffCorrection1972} that the original differential correction method is converging quadratically while the modified versions have linear convergence. Thus, in this section, we wish to present the original version of the differential correction method adapted to our network setting.

This algorithm is also an iterative process. To initialise the differential correction method, we need the following parameters:
\begin{itemize}
	\item Initialisation for the rational function.
	
	One can substitute any value of the decision variables, ${\bf W}$ and $b$, since this method gurantees the convergence from any initial point as long as the denominator is positive. Thereofore, we suggest to use the simplest possible initialisation which is ${\bf W} = 0$ and $b=0$. Then, the initial ratio will be $a_0/b_0$. By using this initial ratio, the maximal absolute deviation $\Delta$ can be found.
	
	\item The absolute precision for the maximal deviation $\varepsilon$.
	
\end{itemize}

Note that our decision variables are now ${\bf W}$ and $b$. At the $k$th step, $P_k({\bf W}x^i+b)$ and $Q_k({\bf W}x^i+b)$ are determined by minimising the auxiliary expression
\begin{equation} \label{eq:ODC}
	\max_{i=1:N} \left \{  \frac{\left |y^i Q_k({\bf W}x^i+b) - P_k({\bf W}x^i+b) \right | - \Delta_{k-1} Q({\bf W}x^i+b)}{Q_{k-1}({\bf W}x^i+b)} \right \}
\end{equation}
subject to the normalisation condition of $\max |w_j,b| = 1, j=1:n$ where $\Delta_{k-1}$ is the maximum absolute error from the previous step. Without loss of generality, the normalisation condition can be substituted by $-1 \leq w_j \leq 1, -1 \leq b \leq 1$ for all $j=1:n$~\cite{DiffCorrection1972}.

This auxiliary function in~(\ref{eq:ODC}) automatically maintains the constraint $Q({\bf W}x^i+b) >0, i=1:N$ and hence it need not to be incorporated into the linear programming problem formulation.

Problem~(\ref{eq:ODC}) can be formulated as follows:
\begin{equation}
	\min \bar{z}
\end{equation}
subject to
\begin{equation}
	y^i Q_k({\bf W}x^i+b) - P_k({\bf W}x^i+b) - \Delta_{k-1} Q({\bf W}x^i+b) \leq \bar{z}Q_{k-1}({\bf W}x^i+b),~ j=1:N
\end{equation}
\begin{equation}
	P_k({\bf W}x^i+b) - y^i Q_k({\bf W}x^i+b) - \Delta_{k-1} Q({\bf W}x^i+b)\leq \bar{z}Q_{k-1}({\bf W}x^i+b),~ j=1:N 
\end{equation}
\begin{equation}
	-1 \leq w_j,b \leq 1, \quad j=1:n.
\end{equation}

Since the constraints are linear with respect to its decision variables, one can use linear programming techniques to find the solution to this problem. This procedure needs to be repeated until $|\Delta_{k-1} - \Delta_k| \leq \epsilon$.

\section{Numerical experiments} \label{sec:numerical_experiments}

\subsection{Fixing the coefficient of the rational activation function}

The coefficients of our one degree rational activation function need to be fixed before we start the optimisation process of the neural network. In this section, we show the best rational approximations to ReLU and LReLU functions computed via bisection method and differential correction algorithm. The precision for both methods is set at $\epsilon=10^{-5}$. The computed coefficients were then used for the experiments in the next section.

Table~\ref{tab:ReLu} includes the computed coefficients of the rational $(1,1)$ approximation of the ReLU function by bisection and differential correction algorithm and the coefficients are very similar for both methods up to the 5th decimal place. This is not surprising since the main objective of both methods is the same. Figure~\ref{fig:ReLU_App_Err} shows the approximations computed by both methods side by side with the error curves. One can see that the approximations and the error curves are identical for both methods.

\begin{table}
	\centering
	\begin{tabular}{|c|c|c| }
		\hline
		Coefficients & Bisection method & Diff correction algorithm \\ 
		\hline
		$a_0$ & 0.118033131217831 & 0.118032919266854 \\
		\hline  
		$a_1$ & 0.309015630106715 & 0.309014524918010 \\
		\hline  
		$b_0$ & 1 & 1 \\
		\hline  
		$b_1$ & -0.618035002222233 & -0.618036788697690 \\
		\hline
		Absolute error & 0.118033601638151 & 0.118032919266855\\
		\hline 
	\end{tabular}
	\caption{The coefficients of the best rational approximation to ReLU function.} 
	\label{tab:ReLu}
\end{table}

\begin{figure}
	\centering
	\subfloat[Rational approximations.]{\label{fig:ReLUapp}\includegraphics[width=.49\textwidth]{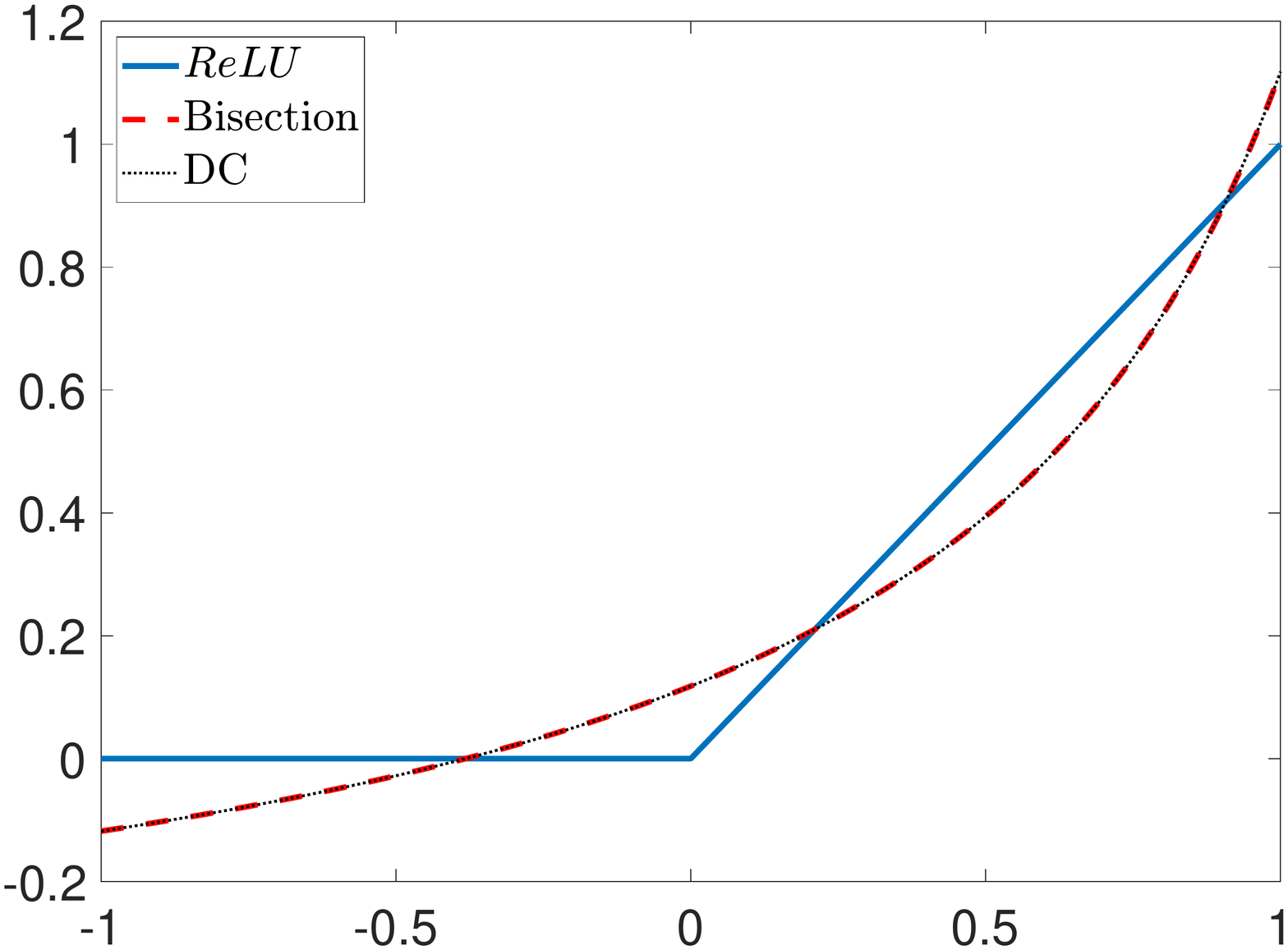}}
	\subfloat[Error curves.]{\label{fig:ReLUerr}\includegraphics[width=.49\textwidth]{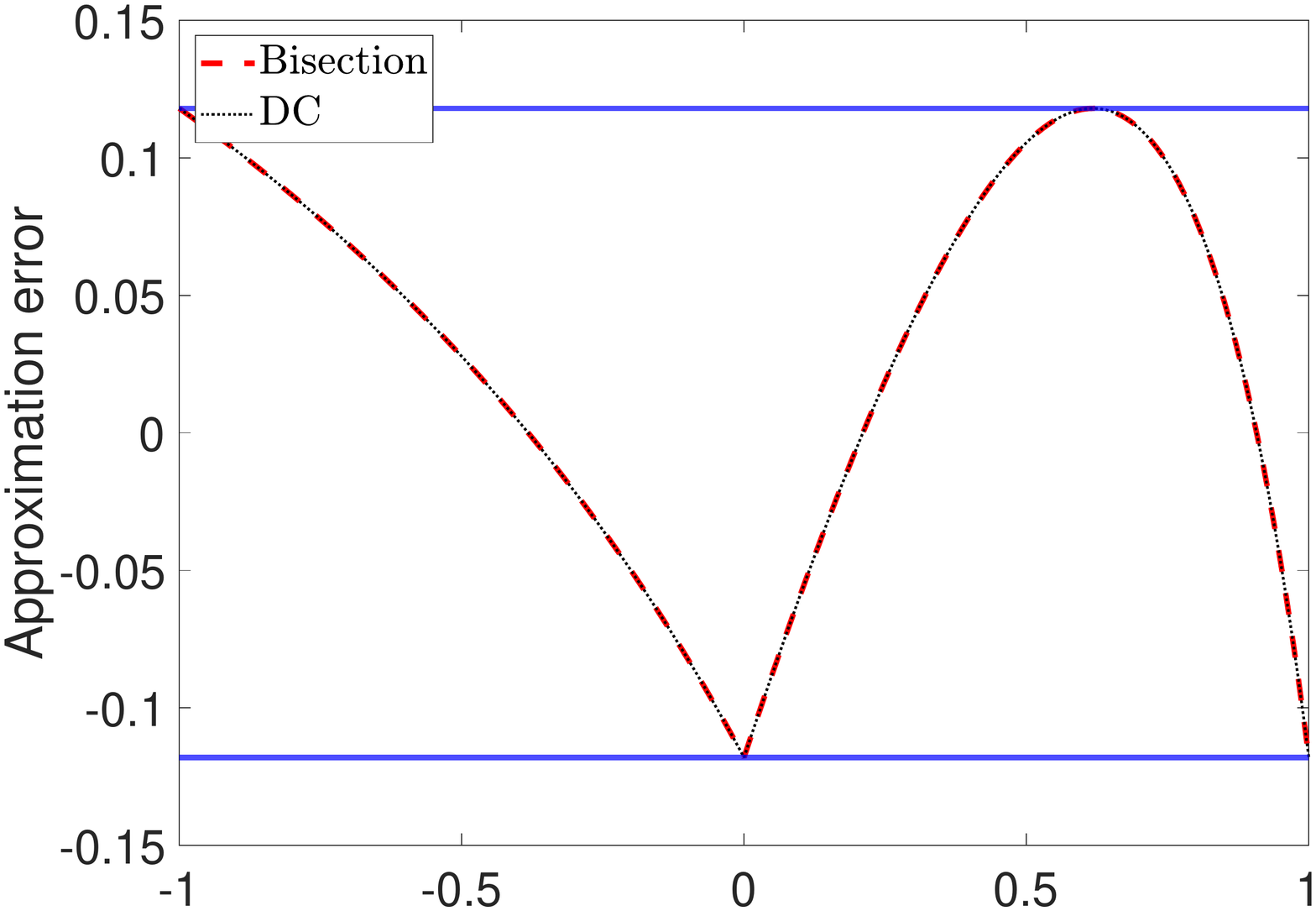}}
	\caption{Rational $(1,1)$ approximations computed by bisection method and differential correction algorithm and its error curves for the ReLU function.}
	\label{fig:ReLU_App_Err}
\end{figure}

Table~\ref{tab:LReLu} includes the computed coefficients of the rational $(1,1)$ approximation of the LReLU function by bisection and differential correction algorithm. Figure~\ref{fig:ReLU_App_Err} shows the approximations computed by both methods side by side with the error curves.

\begin{table}
	\centering
	\begin{tabular}{|c|c|c| }
		\hline
		Coefficients & Bisection method & Diff correction algorithm \\ 
		\hline
		$a_0$ & 0.115809696698241 & 0.115808917873690 \\
		\hline  
		$a_1$ & 0.318466101113600 & 0.318462124442619 \\
		\hline  
		$b_0$ & 1 & 1 \\
		\hline  
		$b_1$ & -0.610795142595234 & -0.610801602891053 \\
		\hline 
		Absolute error & 0.115811407705668 & 0.115808917873690\\
		\hline
	\end{tabular}
	\caption{The coefficients of the best rational approximation to LReLU function.} 
	\label{tab:LReLu}
\end{table}

\begin{figure}
	\centering
	\subfloat[Rational approximations.]{\label{fig:LReLUapp}\includegraphics[width=.49\textwidth]{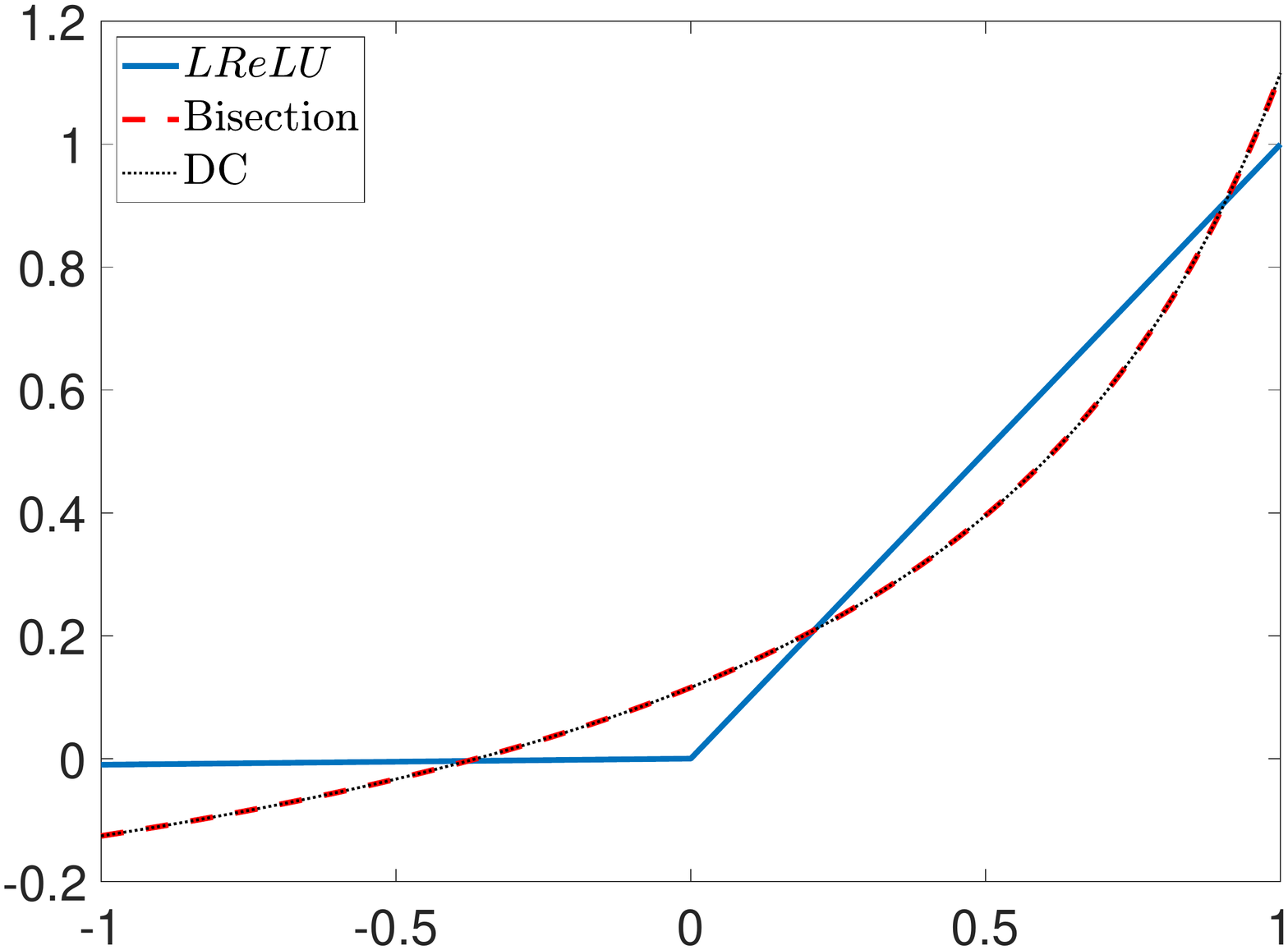}}
	\subfloat[Error curves.]{\label{fig:LReLUerr}\includegraphics[width=.49\textwidth]{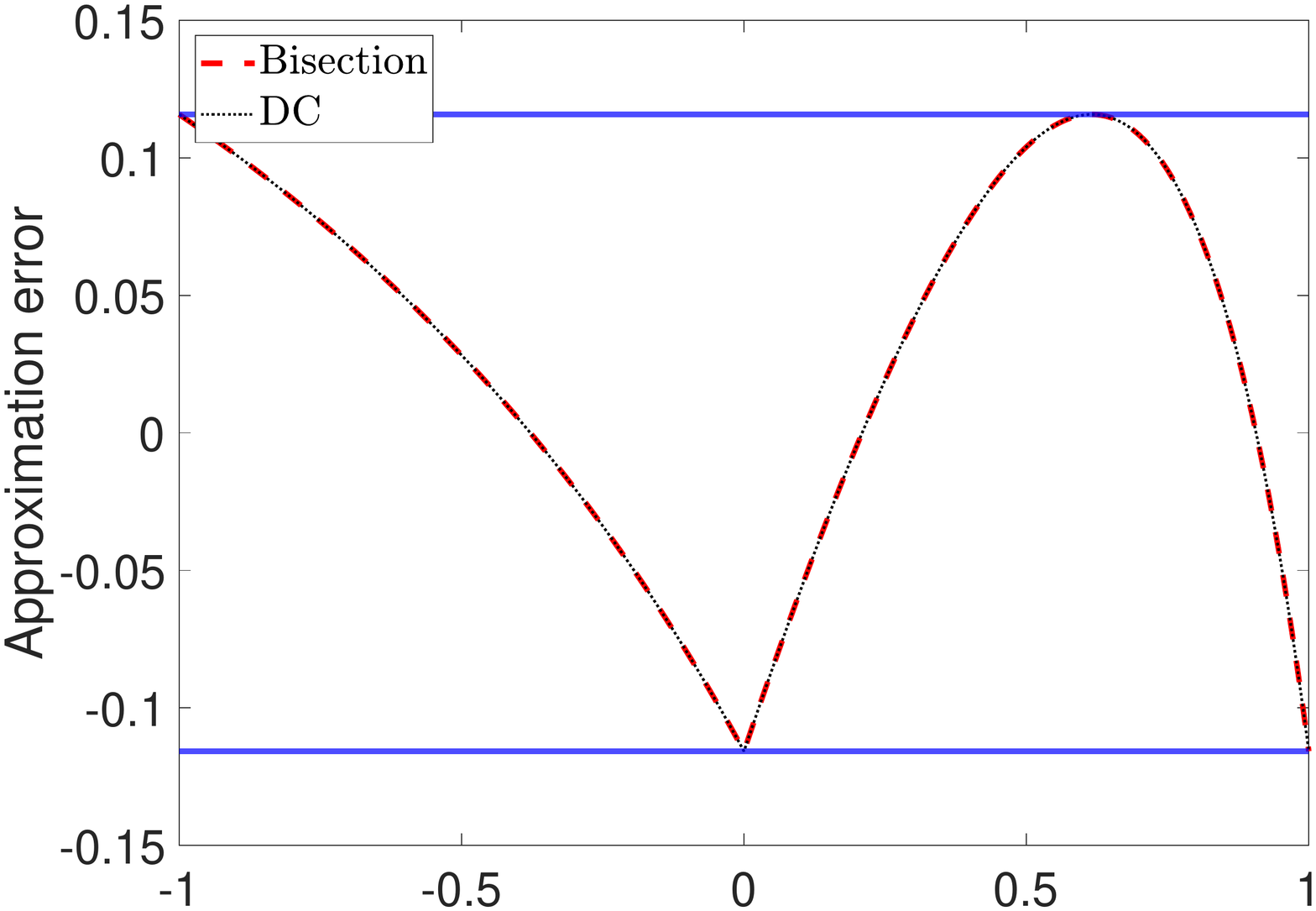}}
	\caption{Rational $(1,1)$ approximations computed by bisection method and differential correction algorithm and its error curves for LReLU function.}
	\label{fig:LReLU_App_Err}
\end{figure}

There are few things that one could notice. The difference between the sets of coefficients computed by the bisection methods and the differential correction algorithm is very small and it is at least within the precision~$\epsilon=10^{-3}$, but for most cases it is within the precision~$\epsilon=10^{-5}$. Even though we use a different normalisation condition for the differential correction algorithm, the coefficient of the smallest degree monomial turned out to be $1$. This observation is not surprising, as the normalisation condition forces the maximum absolute value of the coefficients in the denominator to be less than $1$. Clearly, The approximations are smooth and shape of the approximations are not very different for both ReLU and LReLU functions. However, the sets of coefficients are slightly different as expected. Both error curves comprise of $4$ maximal and minimal alternating points which gurantee that the current approximations are optimal.

Since the ReLU coefficients lead us to better classification accuracy, we only report the results of the experiments run with ReLU coefficients.

\subsection{Experiments with datasets}

The goal of this section is to compare the classification accuracy obtained by our simple neural network when the bisection method and the differential correction algorithm are employed to optimise the parameters of the network along with the classification accuracy obtained by the standard MATLAB toolbox. We intentionaly use datasets with limited training data or imbalanced class data since the uniform approximation based loss function works better on these types of sets.

\subsubsection{Experiments with TwoLeadECG dataset}

We begin our numerical experiments with one of the datasets from the famous PhysioNet~\cite{PhysioNet} database and we only focus on MIT-BIH Long-Term ECG Database. This MIT-BIH Long-Term ECG data collection contains 7 long-term ECG recordings, with manually reviewed beat annotations. We use the dataset, TwoLeadECG which is the seventh set of recordings from the MIT-BIH Long-Term ECG data base.

It contains two classes of signal; class 1 includes signals of type signal 0 (Figure~\ref{fig:signal0}) and class 2 consist of signals of type signal 1 (Figure~\ref{fig:signal1}). The main goal is to distinguish the signals between these two classes. 

The training set contains 23 recordings; 12 recordings from class 1, 11 recordings from class 2 and the test set contains 1139 recordings; 569 recordings from class 1, 570 recordings from class 2.

\begin{figure}
	\centering
	\subfloat[Signal 0.]{\label{fig:signal0}\includegraphics[width=.49\textwidth]{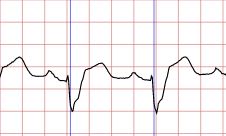}}
	\hspace{1mm}
	\subfloat[Signal 1.]{\label{fig:signal1}\includegraphics[width=.44\textwidth]{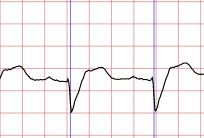}}
	\caption{Signals from the dataset TwoLeadECG data.}
	\label{fig:Signals}
\end{figure}

We use the default activation functions (softmax on the output) for MATLAB deep learning toolbox whose loss function is based on the least squares formula and we use one degree rational activation function for the networks with uniform based loss functions. We report the classification accuracy, confusion matrix, time to train the network and time to calculate the classification accuracy for the test set. We also report the loss function (objective function) values of train and test sets obtained by the bisection method and the differential correction method. 

A confusion matrix, also known as an error matrix, is a table layout that allows one to see how well a classifcation algorithm performs. The main diagonal of the confusion matrix corresponds to the correctly classified points where as the off-diagonal corresponds to misclassified points. Each row of the matrix corresponds to the actual class and each column of the matrix corresponds to the predicted class. Hence, the top left element of the matrix includes the correctly classified points from class one, top right element includes the missclassified points from class 1. 

\subsubsection{Experiments with the original dataset}
We start with the original dataset. Results of the experiments are shown in Table~\ref{tab:ClassOriginal} and Table~\ref{tab:LossOriginal}. 

\begin{table}[htbp]
	\centering
	\caption{Classification results for the original dataset.}
	\begin{tabular}{|c|c|c|c|}
		\hline
		\multirow{2}{*}{Method} & \multirow{2}[-2]{*}{MATLAB} & \multicolumn{2}{c|}{Uniform approximation based loss} \\
		\cline{3-4}          & Toolbox      & \multicolumn{1}{c|}{Bisection} & \multicolumn{1}{c|}{Diff Correction} \\
		\hline
		\multirow{2}[-1]{*}{Test classification} &  \multirow{2}{*}{$70.2\%$}     &  \multirow{2}{*}{$87.71\%$}     & \multirow{2}{*}{$55.31\%$}  \\
		accuracy & & &  \\
		\hline
		%\cline{1-4}
		\multirow{2}[-1]{*}{Confusion matrix} 
		& \begin{tabular}{c|c}
			510  & 59 \\
			\hline
			280  & 290
		\end{tabular}
		
		& \begin{tabular}{c|c}
			535  & 34 \\
			\hline
			106  & 464
		\end{tabular}
		
		& \begin{tabular}{c|c}
			462  & 107 \\
			\hline
			402  & 168
		\end{tabular} \\
		\hline
		\multirow{2}[-2]{*}{Training time} & \multirow{2}[8]{*}{1.056093}      &  \multirow{2}[2]{*}{0.355755}     & \multirow{2}[2]{*}{104.523418} \\
		(in seconds) & & & \\
		\cline{1-1}\cline{3-4}
		\multirow{2}[-2]{*}{Testing time} &       & \multirow{2}[2]{*}{3.436044}      & \multirow{2}[2]{*}{3.994010} \\
		(in seconds) & & & \\
		\hline		
	\end{tabular}%
	\label{tab:ClassOriginal}%
\end{table}%

One can clearly see from Table~\ref{tab:ClassOriginal} that the bisection method is more accurate than the MATLAB toolbox and the differential correction algoritm. Classification accuracy obtained by the differential correction algorithm is the lowest. This may be due to the fact that the normalisation condition of the differential correction algorithm forces all the parameters of the network to be within the interval $[-1,1]$ resulting in underfitting of the model. Another reason could be that the differential correction method stops far from the optimal solution, since in general, it comes to a stop when the improvement in the descent direction is insufficient. Moreover, the differential correction algorithm takes more than a minute to train the network while bisection method takes less than a second.

\begin{table} [htbp]
	\centering
	\caption{Loss functions value of the original dataset.}
	\begin{tabular}{|c|c|c|}
		\hline
		Loss function value & Bisection & Diff correction\\
		\hline
		Training set & $1.345765\times 10^{-5}$  & $1.232949$ \\
		\hline
		Testing set & $40.665565$    & $1$ \\
		\hline
	\end{tabular}
	\label{tab:LossOriginal}
\end{table}

Table~\ref{tab:LossOriginal} contains the objective function values ccomputed by uniform based loss functions. For the training set, the bisection method achieves a smaller loss function (objective function) value than the differential correction method. Because there are no optimisation techniques employed when computing the classification accuracy for the test set, one can only compare the loss function values of the training set from an optimisation standpoint. A lower loss function value for the training set computed by the bisection method implies that it perfectly seperates the training data into the two classes. 

For the test set, the bisection method computes a higher loss function value. One possible explanation for the unusual high loss function value is overfitting of the model. Overfitting may happen due to the presence of a few extreme outliers in the test set which makes the loss function value much higher. To check this claim we remove some extreme points (outliers) from the test set. In particular, we remove the points whose absolute deviation is greater than a predefined threshlod ($\varepsilon$). When $\varepsilon = 10$, $18$ points from the test set were removed and classification accuracy was obtained. Accuracy was then slightly reduced to $87.51\%$ and the loss function value is reduced to $9.8705$, however, bisection still performs better. This reduction of the accuracy is not surprisng since the removal of the outliers reduces the overfitting of the model for the test set. 

We take one step further and set $\varepsilon = 5$. This reduces the test set to $1084$ samples by removing 55 outliers from the set. The classification accuracy is slightly reduced again to $87.18$ and the loss function value is now $4.8820$. This confirms that the overfitting of the model for the test set does not greatly affect the classification accuracy.

In general, we expect the bisection method and the differential correction algorithm to perform similarly, however, when it comes to training a simple neural network, these two methods perform differently in terms of the classification accuracy. This needs to be investigated further to find the reason behind this observation and we leave it for future research.

\subsubsection{Experiments with reduced dataset}
Now we make our training set even much smaller by randomly selecting 10 points from the original training set and compare the classification accuracy against the test set. The classification results can be found in Table~\ref{tab:ClassRandom}.

\begin{table}[htbp]
	\centering
	\caption{Classification results for the randomly selected points in the training dataset.}
	\begin{tabular}{|c|c|c|c|}
		\hline
		\multirow{2}{*}{Method} & \multirow{2}[-2]{*}{MATLAB} & \multicolumn{2}{c|}{Uniform approximation based loss} \\
		\cline{3-4}          & Toolbox      & \multicolumn{1}{c|}{Bisection} & \multicolumn{1}{c|}{Diff Correction} \\
		\hline
		\multirow{2}[-1]{*}{Test classification} &  \multirow{2}{*}{$65.5\%$}     &  \multirow{2}{*}{$79.10\%$}     & \multirow{2}{*}{$53.64\%$}  \\
		accuracy & & &  \\
		\hline
		%\cline{1-4}
		\multirow{2}[-1]{*}{Confusion matrix} 
		& \begin{tabular}{c|c}
			209  & 306 \\
			\hline
			33  & 537
		\end{tabular}
		
		& \begin{tabular}{c|c}
			504  & 65 \\
			\hline
			173  & 397
		\end{tabular}
		
		& \begin{tabular}{c|c}
			565  & 4 \\
			\hline
			524  & 46
		\end{tabular} \\
		\hline
		\multirow{2}[-1]{*}{Training time} & \multirow{2}[8]{*}{3.266344}      &  \multirow{2}[2]{*}{1.861076}     & \multirow{2}[2]{*}{34.267841} \\
		(in seconds) & & & \\
		\cline{1-1}\cline{3-4}
		\multirow{2}[-1]{*}{Testing time} &       & \multirow{2}[2]{*}{3.673801}      & \multirow{2}[2]{*}{4.897909} \\
		(in seconds) & & & \\
		\hline		
	\end{tabular}%
	\label{tab:ClassRandom}%
\end{table}%

Clearly, from Table~\ref{tab:ClassRandom}, the classification accuracy for all three methods decreased. However, the bisection method outperforms both MATLAB toolbox and the differential correction method while demanding less time than the differential correction method to train and test the model.

\begin{table} [htbp]
	\centering
	\caption{Loss functions value of the reduced dataset.}
	\begin{tabular}{|c|c|c|}
		\hline
		Loss function value & Bisection & Diff correction\\
		\hline
		Training set & $1.345765\times 10^{-5}$  & $1.225767$ \\
		\hline
		Testing set & $30.5$    & $1$ \\
		\hline
	\end{tabular}
	\label{tab:LossRandom}
\end{table}

Table~\ref{tab:LossRandom} contains the objective function values. When the bisection method is employed, the loss function value for the training set is the same as the loss function value for the original training dataset. There is a small decreament in the loss function value of the differential correction algorithm for the training set compared to the original training set. Loss function value of the bisection method for testing set is reduced by $10$ units while it remains the same for the differential correction method.

Now we remove some extreme points from the test set whose absolute deviation is maximal or close to maximal where the threshold is set at $5$. In this case, $54$ points were removed from the original test set. The classification accuracy is then reduced to $78.06\%$ with the loss function value of $4.8957$. Bisection method is perform better than the other two methods.

\subsubsection{Experiments with imbalaned distribution between classes}
Our next step is to consider imbalanced distribution of the data points between classes. This means that one of the two classes is underrepresented in the corresponding training dataset. We consider two cases: in the first case, class 1 is underrepresented and in the case two, class 2 is underrepresented. 

{\bf Case 1:} Training set consist of all the points from class 2 with $2$ randomly selected points from class 1. Size of the training set is now $13 (11+2)$. The classification accuracy is given in Table~\ref{tab:ClassUnderClass1}. 

\begin{table}[htbp]
	\centering
	\caption{Classification results when class 1 is underrepresented in the training set.}
	\begin{tabular}{|c|c|c|c|}
		\hline
		\multirow{2}{*}{Method} & \multirow{2}[-2]{*}{MATLAB} & \multicolumn{2}{c|}{Uniform approximation based loss} \\
		\cline{3-4}          & Toolbox      & \multicolumn{1}{c|}{Bisection} & \multicolumn{1}{c|}{Diff Correction} \\
		\hline
		\multirow{2}[-1]{*}{Test classification} &  \multirow{2}{*}{$53.6\%$}     &  \multirow{2}{*}{$66.54\%$}     & \multirow{2}{*}{$53.99\%$}  \\
		accuracy & & &  \\
		\hline
		%\cline{1-4}
		\multirow{2}[-1]{*}{Confusion matrix} 
		& \begin{tabular}{c|c}
			519  & 50 \\
			\hline
			92  & 478
		\end{tabular}
		
		& \begin{tabular}{c|c}
			300  & 269 \\
			\hline
			112  & 458
		\end{tabular}
		
		& \begin{tabular}{c|c}
			531  & 38 \\
			\hline
			486  & 84
		\end{tabular} \\
		\hline
		\multirow{2}[-1]{*}{Training time} & \multirow{2}[8]{*}{3.304957}      &  \multirow{2}[2]{*}{2.138431}     & \multirow{2}[2]{*}{1.550187} \\
		(in seconds) & & & \\
		\cline{1-1}\cline{3-4}
		\multirow{2}[-1]{*}{Testing time} &       & \multirow{2}[2]{*}{3.182606}      & \multirow{2}[2]{*}{2.309288} \\
		(in seconds) & & & \\
		\hline		
	\end{tabular}%
	\label{tab:ClassUnderClass1}%
\end{table}%

From Table~\ref{tab:ClassUnderClass1}, we see that the classification accuracy for all three method decreased. However, the accuracy calculated using the bisection method, continues to be superior. At the same time, it takes much longer to train the network and to test than the differential correction algorithm.

\begin{table} [htbp]
	\centering
	\caption{Loss functions value when class 1 is underrepresented.}
	\begin{tabular}{|c|c|c|}
		\hline
		Loss function value & Bisection & Diff correction\\
		\hline
		Training set & $1.345765\times 10^{-5}$  & $1.255002$ \\
		\hline
		Testing set & $3.300952$    & $13.3$ \\
		\hline
	\end{tabular}
	\label{tab:LossUnderClass1}
\end{table}

Table~\ref{tab:LossUnderClass1} consists of the objective function values. When the bisection method is employed, the value of the loss function computed for the test set has drastically been reduced while differential correction method achieves a higher loss function value. 

{\bf Case 2:} Training set consist of all the points from class 1 with $2$ randomly selected points from class 2. Size of the training set is now $14 (12+2)$. The classification accuracy is given in Table~\ref{tab:ClassUnderClass2}. 

\begin{table}[htbp]
	\centering
	\caption{Classification results when class 2 is underrepresented in the training set.}
	\begin{tabular}{|c|c|c|c|}
		\hline
		\multirow{2}{*}{Method} & \multirow{2}[-2]{*}{MATLAB} & \multicolumn{2}{c|}{Uniform approximation based loss} \\
		\cline{3-4}          & Toolbox      & \multicolumn{1}{c|}{Bisection} & \multicolumn{1}{c|}{Diff Correction} \\
		\hline
		\multirow{2}[-1]{*}{Test classification} &  \multirow{2}{*}{$54.7\%$}     &  \multirow{2}{*}{$65.84\%$}     & \multirow{2}{*}{$53.99\%$}  \\
		accuracy & & &  \\
		\hline
		%\cline{1-4}
		\multirow{2}[-1]{*}{Confusion matrix} 
		& \begin{tabular}{c|c}
			63  & 506 \\
			\hline
			10  & 560
		\end{tabular}
		
		& \begin{tabular}{c|c}
			566  & 3 \\
			\hline
			386  & 184
		\end{tabular}
		
		& \begin{tabular}{c|c}
			566  & 3 \\
			\hline
			521  & 49
		\end{tabular} \\
		\hline
		\multirow{2}[-1]{*}{Training time} & \multirow{2}[8]{*}{3.358267}      &  \multirow{2}[2]{*}{1.374672}     & \multirow{2}[2]{*}{1.234968} \\
		(in seconds) & & & \\
		\cline{1-1}\cline{3-4}
		\multirow{2}[-1]{*}{Testing time} &       & \multirow{2}[2]{*}{3.868916}      & \multirow{2}[2]{*}{3.638340} \\
		(in seconds) & & & \\
		\hline	
	\end{tabular}%
	\label{tab:ClassUnderClass2}%
\end{table}%

The accuracy calculated using the bisection method, continues to be superior and the bisection method and the differential correction algorithm take roughly the same amount of time to train and test.

\begin{table} [htbp]
	\centering
	\caption{Loss functions value when class 2 is underrepresented.}
	\begin{tabular}{|c|c|c|}
		\hline
		Loss function value & Bisection & Diff correction\\
		\hline
		Training set & $1.345765\times 10^{-5}$  & $0.370170$ \\
		\hline
		Testing set & $13.177499$    & $269.4$ \\
		\hline
	\end{tabular}
	\label{tab:LossUnderClass2}
\end{table}

From Table~\ref{tab:LossUnderClass2}, one can see that the differential correction method achieves smaller loss function values for the training set compared to the pervious cases, but for the test set, it is much higher than before. We now remove some extreme points from the test set whose absolute deviation is maximal and we define the threshold, $\varepsilon=5$ for this case. Only $8$ points were removed from the original test set. Then we compute the classification accuracy by using the differential correction method. The classification accuracy is slightly reduced to $53.67\%$ with the loss function value of $4.4433$. One can reduce the threshold to remove more outliers from the test set and reduce the loss function value further.

\subsubsection{Experiments with SonyAIBORobotSurface1 dataset}

We now consider a different dataset from~\cite{SonyAIrobot}. The SONY AIBO Robot is a small, dog-shaped robot equipped with multiple sensors. In the experimental setting, the robot walked on two different surfaces: carpet and cement. Class 1 comprise of the data when the robot walked on the carpet and class 2 consist of the data of the robot when it walked on the cement floor. The main goal is to distinguish the type of floor that the robot walked on. 

The training set contains 20 recordings; 6 recordings from class 1, 14 recordings from class 2 and the test set contains 601 recordings; 343 recordings from class 1, 258 recordings from class 2.

\subsubsection{Experiments with the original dataset}
We start with the original dataset. Results of the experiments are shown in Table~\ref{tab:ClassOriginal2} and Table~\ref{tab:LossOriginal2}. 

% Table generated by Excel2LaTeX from sheet 'Sheet1'
\begin{table}[htbp]
	\centering
	\caption{Classification results for the original dataset.}
	\begin{tabular}{|c|c|c|c|}
		\hline
		\multirow{2}{*}{Method} & \multirow{2}[-2]{*}{MATLAB} & \multicolumn{2}{c|}{Uniform approximation based loss} \\
		\cline{3-4}          & Toolbox      & \multicolumn{1}{c|}{Bisection} & \multicolumn{1}{c|}{Diff Correction} \\
		\hline
		\multirow{2}[-1]{*}{Test classification} &  \multirow{2}{*}{$58.9\%$}     &  \multirow{2}{*}{$67.22\%$}     & \multirow{2}{*}{$63.06\%$}  \\
		accuracy & & &  \\
		\hline
		%\cline{1-4}
		\multirow{2}[-1]{*}{Confusion matrix} 
		& \begin{tabular}{c|c}
			101  & 242 \\
			\hline
			5  & 253
		\end{tabular}
		
		& \begin{tabular}{c|c}
			235  & 108 \\
			\hline
			89  & 169
		\end{tabular}
		
		& \begin{tabular}{c|c}
			317  & 26 \\
			\hline
			196  & 62
		\end{tabular} \\
		\hline
		\multirow{2}[-2]{*}{Training time} & \multirow{2}[8]{*}{1.445051}      &  \multirow{2}[2]{*}{0.706029}     & \multirow{2}[2]{*}{13.420159} \\
		(in seconds) & & & \\
		\cline{1-1}\cline{3-4}
		\multirow{2}[-2]{*}{Testing time} &       & \multirow{2}[2]{*}{1.841726}      & \multirow{2}[2]{*}{2.224784} \\
		(in seconds) & & & \\
		\hline		
	\end{tabular}%
	\label{tab:ClassOriginal2}%
\end{table}%

One can clearly see that the bisection method is more accurate than the MATLAB toolbox and the differential correction algoritm. In this case, classification accuracy obtained by the MATLAB toolbox is the lowest. Bisection method still takes less time to train and test the model than the differential correction method. 

\begin{table} [htbp]
	\centering
	\caption{Loss functions value for the original dataset.}
	\begin{tabular}{|c|c|c|}
		\hline
		Loss function value & Bisection & Diff correction\\
		\hline
		Training set & $1.3142\times 10^{-8}$  & $1.2267$ \\
		\hline
		Testing set & $920.52$    & $1.1968$ \\
		\hline
	\end{tabular}
	\label{tab:LossOriginal2}
\end{table}

The loss function value computed by the bisection method for the training set is much smaller compared to the differential correction method and vise versa for the test set. 

Now we remove some extreme points (outliers) from the test set where $\varepsilon=5$. In particular, $83$ points were removed from the original test set. Then the classification accuracy computed by the bisection method for the test set is slightly reduced to $66.80\%$ where the loss function value is $4.5420$. Bisection method still achieves the highest classification accuracy.

\subsubsection{Experiments with reduced dataset}
Now we make our training set smaller by randomly selecting 10 points from the original training set and compare the classification accuracy against the test set. The classification results can be found in Table~\ref{tab:ClassRandom2}. 

\begin{table}[htbp]
	\centering
	\caption{Classification results for the randomly selected points in the training dataset.}
	\begin{tabular}{|c|c|c|c|}
		\hline
		\multirow{2}{*}{Method} & \multirow{2}[-2]{*}{MATLAB} & \multicolumn{2}{c|}{Uniform approximation based loss} \\
		\cline{3-4}          & Toolbox      & \multicolumn{1}{c|}{Bisection} & \multicolumn{1}{c|}{Diff Correction} \\
		\hline
		\multirow{2}[-1]{*}{Test classification} &  \multirow{2}{*}{$54.3\%$}     &  \multirow{2}{*}{$66.89\%$}     & \multirow{2}{*}{$57.74\%$}  \\
		accuracy & & &  \\
		\hline
		%\cline{1-4}
		\multirow{2}[-1]{*}{Confusion matrix} 
		& \begin{tabular}{c|c}
			68  & 274 \\
			\hline
			0  & 258
		\end{tabular}
		
		& \begin{tabular}{c|c}
			190  & 153 \\
			\hline
			46  & 212
		\end{tabular}
		
		& \begin{tabular}{c|c}
			342  & 1 \\
			\hline
			253  & 5
		\end{tabular} \\
		\hline
		\multirow{2}[-1]{*}{Training time} & \multirow{2}[8]{*}{1.490741}      &  \multirow{2}[2]{*}{2.119713}     & \multirow{2}[2]{*}{11.368623} \\
		(in seconds) & & & \\
		\cline{1-1}\cline{3-4}
		\multirow{2}[-1]{*}{Testing time} &       & \multirow{2}[2]{*}{2.171950}      & \multirow{2}[2]{*}{1.965314} \\
		(in seconds) & & & \\
		\hline		
	\end{tabular}%
	\label{tab:ClassRandom2}%
\end{table}%

All three techniques have lower classification accuracy than the original dataset. However, the decrement of the accuracy of the bisection method is very small compared to the other two methods, at the same time, accuracy obtained by the bisection method and the differential correction method is not too far apart. Bisection method takes about the same amount of time to train and test the model whereas the differential correction method takes less time to test the model than the bisection method.

\begin{table} [htbp]
	\centering
	\caption{Loss functions value for the reduced dataset.}
	\begin{tabular}{|c|c|c|}
		\hline
		Loss function value & Bisection & Diff correction\\
		\hline
		Training set & $1.3458\times 10^{-5}$  & $1.1854$ \\
		\hline
		Testing set & $19.3$    & $1$ \\
		\hline
	\end{tabular}
	\label{tab:LossRandom2}
\end{table}

Table~\ref{tab:LossRandom2} contains the objective function values of the uniform based loss functions. The value of the loss function for the training set is much smaller for the bisection method, but for the test test, it is much higher compared to the differential correction method.

\subsubsection{Experiments with imbalaned distribution between classes}
Now we consider imbalanced distribution of the data points between classes. 
%This means that one of the two classes is underrepresented in the corresponding training dataset. We consider two cases: in the first case, class 1 is underrepresented and in the case two, class 2 is underrepresented. 

{\bf Case 1:} Training set consist of all the points from class 2 with $2$ randomly selected points from class 1. Size of the training set is now $13 (11+2)$. The classification accuracy is given in Table~\ref{tab:ClassUnderClass12}. 

\begin{table}[htbp]
	\centering
	\caption{Classification results when class 1 is underrepresented in the training set.}
	\begin{tabular}{|c|c|c|c|}
		\hline
		\multirow{2}{*}{Method} & \multirow{2}[-2]{*}{MATLAB} & \multicolumn{2}{c|}{Uniform approximation based loss} \\
		\cline{3-4}          & Toolbox      & \multicolumn{1}{c|}{Bisection} & \multicolumn{1}{c|}{Diff Correction} \\
		\hline
		\multirow{2}[-1]{*}{Test classification} &  \multirow{2}{*}{$52.2\%$}     &  \multirow{2}{*}{$73.54\%$}     & \multirow{2}{*}{$65.06\%$}  \\
		accuracy & & &  \\
		\hline
		%\cline{1-4}
		\multirow{2}[-1]{*}{Confusion matrix} 
		& \begin{tabular}{c|c}
			77  & 266 \\
			\hline
			3  & 255
		\end{tabular}
		
		& \begin{tabular}{c|c}
			220  & 123 \\
			\hline
			36  & 222
		\end{tabular}
		
		& \begin{tabular}{c|c}
			255  & 88 \\
			\hline
			122  & 136
		\end{tabular} \\
		\hline
		\multirow{2}[-1]{*}{Training time} & \multirow{2}[8]{*}{1.308744}      &  \multirow{2}[2]{*}{1.928491}     & \multirow{2}[2]{*}{1.485842} \\
		(in seconds) & & & \\
		\cline{1-1}\cline{3-4}
		\multirow{2}[-1]{*}{Testing time} &       & \multirow{2}[2]{*}{2.153214}      & \multirow{2}[2]{*}{1.985884} \\
		(in seconds) & & & \\
		\hline		
	\end{tabular}%
	\label{tab:ClassUnderClass12}%
\end{table}%

Clearly, the classification accuracy computed by the MATLAB toolbox decreased, but the accuracy for the uniform based loss functions increased. The accuracy calculated using the bisection method, continues to be superior. At the same time, it takes longer to test the model than the differential correction algorithm. 

\begin{table} [htbp]
	\centering
	\caption{Loss functions value when class 1 is underrepresented.}
	\begin{tabular}{|c|c|c|}
		\hline
		Loss function value & Bisection & Diff correction\\
		\hline
		Training set & $1.3458\times 10^{-5}$  & $0.4259$ \\
		\hline
		Testing set & $205.722$    & $1$ \\
		\hline
	\end{tabular}
	\label{tab:LossUnderClass12}
\end{table}

From Table~\ref{tab:LossUnderClass12}, one can see that the value of the loss function computed by the bisection method for the training set continues to be the same as the previous cases and the loss function value for the test set is higher in general.

We remove some extreme points from the test set whose absolute deviation is maximal or close to maximal by setting $\varepsilon=5$. In this case, $62$ points were removed from the original test set. The classification accuracy is slightly reduced to $73.21\%$ where the loss function value is $4.8615$. Bisection method still performs better than the other two methods.

{\bf Case 2:} Training set consist of all the points from class 1 with 2 randomly selected points from class 2. Size of the training set is now $14 (12+2)$. The classification accuracy is given in Table~\ref{tab:ClassUnderClass22}. 

\begin{table}[htbp]
	\centering
	\caption{Classification results when class 2 is underrepresented in the training set.}
	\begin{tabular}{|c|c|c|c|}
		\hline
		\multirow{2}{*}{Method} & \multirow{2}[-2]{*}{MATLAB} & \multicolumn{2}{c|}{Uniform approximation based loss} \\
		\cline{3-4}          & Toolbox      & \multicolumn{1}{c|}{Bisection} & \multicolumn{1}{c|}{Diff Correction} \\
		\hline
		\multirow{2}[-1]{*}{Test classification} &  \multirow{2}{*}{$57.9\%$}     &  \multirow{2}{*}{$66.22\%$}     & \multirow{2}{*}{$61.90\%$}  \\
		accuracy & & &  \\
		\hline
		%\cline{1-4}
		\multirow{2}[-1]{*}{Confusion matrix} 
		& \begin{tabular}{c|c}
			302  & 41 \\
			\hline
			212  & 46
		\end{tabular}
		
		& \begin{tabular}{c|c}
			208  & 135 \\
			\hline
			68  & 190	
		\end{tabular}
		
		& \begin{tabular}{c|c}
			219  & 124 \\
			\hline
			105  & 153
		\end{tabular} \\
		\hline
		\multirow{2}[-1]{*}{Training time} & \multirow{2}[8]{*}{1.511403}      &  \multirow{2}[2]{*}{1.848319}     & \multirow{2}[2]{*}{1.291203} \\
		(in seconds) & & & \\
		\cline{1-1}\cline{3-4}
		\multirow{2}[-1]{*}{Testing time} &       & \multirow{2}[2]{*}{2.036241}      & \multirow{2}[2]{*}{1.965994} \\
		(in seconds) & & & \\
		\hline	
	\end{tabular}%
	\label{tab:ClassUnderClass22}%
\end{table}%

The accuracy calculated using the bisection method, continues to be superior and the bisection method and the differential correction algorithm take roughly the same amount of time to test.

\begin{table} [htbp]
	\centering
	\caption{Loss functions value when class 2 is underrepresented.}
	\begin{tabular}{|c|c|c|}
		\hline
		Loss function value & Bisection & Diff correction\\
		\hline
		Training set & $1.3488\times 10^{-5}$  & $0.2005$ \\
		\hline
		Testing set & $558.2861$    & $1$ \\
		\hline
	\end{tabular}
	\label{tab:LossUnderClass22}
\end{table}

From Table~\ref{tab:LossUnderClass22} which consists of the objective function values, one can see that the value of the loss function computed by the bisection method for the training set is still smaller than that of the differential correction method, but for the test tes, bisection method computes a very large value.

Now we remove some outliers from the test set whose absolute deviation is maximal or close to maximal by setting $\varepsilon=5$. In particular, $32$ points were removed from the original test set. The classification accuracy is slightly reduced to $65.44\%$ where the loss function value is $4.0802$. Bisection method still performs better than the other two methods.

In general, bisection method performs consistently better than the MATLAB toolbox and the differential correction method when training simple neural network in classification problems. In particular, when the activations functions are one degree rational functions.

\section{Conclusions}  \label{sec:conclusions}

In this paper, we used one degree classical rational functions as activation functions in a simple neural network whose loss function is based on uniform approximation. In this setting, we demonstarted that the corresponding optimisation problems appearing in the network form a generalised rational uniform approximation problem when the coefficients of the rational function are fixed. To find the weights and the bias of the network, we used two methods: the bisection methods and the differential correction algorithm. Our numerical experiments performed on classification problems confirm that the bisection method returns superior accuracy results when the training set is either small or imbalanced between the classes.

In the future, we are planning to continue our work in the direction of extending the use of one degree classical rational activation function to a network with one or more hidden layers. This generalisation is possible if we sandwich layers on top of each other since the quasiconvexity property is preserved. We would also like to study the optimisation problems appearing in neural network when the degree of the rational activation function is more than one. Another interesting research direction is to investigate the reasons behind the low accuracy results obtained by the differential correction algorithm even though it has sure convergence properties in general.

\vskip 6mm
\noindent{\bf Acknowledgements}

\noindent
This research was supported by the Australian Research Council (ARC), Solving hard Chebyshev approximation problems through nonsmooth analysis (Discovery Project DP180100602).


\begin{thebibliography}{99}
	
	\bibitem{DiffCorrection1972}
	I.~Barrodale, MJD. Powell, and FDK. Roberts.
	\newblock The differential correction algorithm for rational
	$l_\infty$-approximation.
	\newblock {\em SIAM Journal on Numerical Analysis}, 9(3):493--504, 1972.
	
	\bibitem{bengio1994learning}
	Y.~Bengio, P.~Simard, and P.~Frasconi.
	\newblock Learning long-term dependencies with gradient descent is difficult.
	\newblock {\em IEEE transactions on neural networks}, 5(2):157--166, 1994.
	
	\bibitem{boull2020a}
	N.~Boullé, Y.~Nakatsukasa, and A.~Townsend.
	\newblock Rational neural networks.
	\newblock Conference on Neural Information Processing Systems, 2020.
	
	\bibitem{SL}
	S.~Boyd and L.~Vandenberghe.
	\newblock {\em Convex Optimization}.
	\newblock Cambridge University Press, New York, NY, USA, 2010.
	
	\bibitem{CAO201817}
	C.~Cao, F.~Liu, H.~Tan, D.~Song, W.~Shu, W.~Li, Y.~Zhou, X.~Bo, and Z.~Xie.
	\newblock Deep learning and its applications in biomedicine.
	\newblock {\em Genomics, Proteomics and Bioinformatics}, 16(1):17--32, 2018.
	
	\bibitem{chen2018rational}
	Z.~Chen, F.~Chen, R.~Lai, X.~Zhang, and CT. Lu.
	\newblock Rational neural networks for approximating graph convolution operator
	on jump discontinuities.
	\newblock In {\em 2018 IEEE International Conference on Data Mining (ICDM)},
	pages 59--68. IEEE, 2018.
	
	\bibitem{cheney1961DCoriginal}
	EW. Cheney and HL. Loeb.
	\newblock Two new algorithms for rational approximation.
	\newblock {\em Numerische Mathematik}, 3(1):72--75, 1961.
	
	\bibitem{cheney1962DC}
	EW. Cheney and HL. Loeb.
	\newblock On rational chebyshev approximation.
	\newblock {\em Numerische Mathematik}, 4(1):124--127, 1962.
	
	\bibitem{cheney1964generalized}
	EW. Cheney and HL. Loeb.
	\newblock Generalized rational approximation.
	\newblock {\em Journal of the Society for Industrial and Applied Mathematics,
		Series B: Numerical Analysis}, 1(1):11--25, 1964.
	
	\bibitem{cheney1963survey}
	EW. Cheney and TH. Southard.
	\newblock A survey of methods for rational approximation, with particular
	reference to a new method based on a forumla of darboux.
	\newblock {\em SIAM Review}, 5(3):219--231, 1963.
	
	\bibitem{cheng2018polynomial}
	X.~Cheng, B.~Khomtchouk, N.~Matloff, and P.~Mohanty.
	\newblock Polynomial regression as an alternative to neural nets.
	\newblock {\em arXiv preprint arXiv:1806.06850}, 2018.
	
	\bibitem{DaCruzAlgorithmsQuasiconvex}
	JX. Da~Cruz~Neto, JO. Lopes, and MV. Travaglia.
	\newblock Algorithms for quasiconvex minimization.
	\newblock {\em Optimization}, 60(8-9):1105--1117, 2011.
	
	\bibitem{DaniilidisHadjisavvasMartinezLegas2002}
	A.~Daniilidis, N.~Hadjisavvas, and JE. Martinez-Legaz.
	\newblock An appropriate subdifferential for quasiconvex functions.
	\newblock {\em SIAM Journal on Optimization}, 12:407--420, 2002.
	
	\bibitem{delfosse2021recurrent}
	Q.~Delfosse, P.~Schramowski, A.~Molina, and K.~Kersting.
	\newblock Recurrent rational networks.
	\newblock {\em arXiv preprint arXiv:2102.09407}, 2021.
	
	\bibitem{dutta2005abstract}
	J.~Dutta and AM. Rubinov.
	\newblock Abstract convexity.
	\newblock {\em Handbook of generalized convexity and generalized monotonicity},
	76:293--333, 2005.
	
	\bibitem{fenchel1953}
	W.~Fenchel and DW. Blackett.
	\newblock {\em Convex cones, sets, and functions}.
	\newblock Princeton University, Department of Mathematics, Logistics Research
	Project, 1953.
	
	\bibitem{PhysioNet}
	A.~Goldberger, L.~Amaral, L~. Glass, J.~Hausdorff, PC. Ivanov, R.~Mark, JE.
	Mietus, GB. Moody, CK. Peng, and HE. Stanley.
	\newblock Physiobank, physiotoolkit, and physionet: Components of a new
	research resource for complex physiologic signals., 2000.
	
	\bibitem{goodfellow2016deep}
	I.~Goodfellow, Y.~Bengio, and A.~Courville.
	\newblock {\em Deep learning}.
	\newblock MIT press, 2016.
	
	\bibitem{goyal2019poly}
	M.~Goyal, R.~Goyal, and B.~Lall.
	\newblock Learning activation functions: A new paradigm for understanding
	neural networks.
	\newblock {\em arXiv preprint arXiv:1906.09529}, 2019.
	
	\bibitem{rice1969}
	Rice JR.
	\newblock {\em The Approximation of Functions: Nonlinear and multivariate
		theory}.
	\newblock Addison-Wesley series in computer science and information processing.
	Mass., Addison-Wesley Publishing Company, 1969.
	
	\bibitem{lecun2015deep}
	Y.~LeCun, Y.~Bengio, and G.~Hinton.
	\newblock Deep learning.
	\newblock {\em nature}, 521(7553):436--444, 2015.
	
	\bibitem{loeb1960}
	HL. Loeb.
	\newblock Algorithms for chebyshev approximations using the ratio of linear
	forms.
	\newblock {\em Journal of the Society for Industrial and Applied Mathematics},
	8(3):458--465, 1960.
	
	\bibitem{MLegazquasiconvexduality}
	JE. Martínez-Legaz.
	\newblock Quasiconvex duality theory by generalized conjugation methods.
	\newblock {\em Optimization}, 19(5):603--652, 1988.
	
	\bibitem{mazurowski2008training}
	MA. Mazurowski, PA. Habas, JM. Zurada, JY. Lo, JA. Baker, and GD. Tourassi.
	\newblock Training neural network classifiers for medical decision making: The
	effects of imbalanced datasets on classification performance.
	\newblock {\em Neural networks}, 21(2-3):427--436, 2008.
	
	\bibitem{meltser1996approximating}
	M.~Meltser, M.~Shoham, and LM. Manevitz.
	\newblock Approximating functions by neural networks: a constructive solution
	in the uniform norm.
	\newblock {\em Neural Networks}, 9(6):965--978, 1996.
	
	\bibitem{millan2021multivariate}
	R. D{\'i}az Mill{\'a}n, V.~Peiris, N.~Sukhorukova, and J.~Ugon.
	\newblock Multivariate approximation by polynomial and generalised rational
	functions.
	\newblock {\em arXiv preprint arXiv:2101.11786}, 2021.
	
	\bibitem{bataineh2017}
	B.~Mohammad and M.~Timothy.
	\newblock Neural network for regression problems with reduced training sets.
	\newblock {\em Neural Networks}, 95:1--9, 2017.
	
	\bibitem{molina2019pade}
	A.~Molina, P.~Schramowski, and K.~Kersting.
	\newblock Pad{\'e} activation units: End-to-end learning of flexible activation
	functions in deep networks.
	\newblock In {\em International Conference on Learning Representations}, 2019.
	
	\bibitem{trefethen2018}
	Y.~Nakatsukasa, O.~S{\`e}te, and LN. Trefethen.
	\newblock The aaa algorithm for rational approximation.
	\newblock {\em SIAM Journal on Scientific Computing}, 40(3):A1494--A1522, 2018.
	
	\bibitem{oh2003polynomial}
	SK. Oh, W.~Pedrycz, and BJ. Park.
	\newblock Polynomial neural networks architecture: analysis and design.
	\newblock {\em Computers \& Electrical Engineering}, 29(6):703--725, 2003.
	
	\bibitem{AMCPeirisSukhSharonUgon}
	V.~Peiris, N.~Sharon, N.~Sukhorukova, and J.~Ugon.
	\newblock Generalised rational approximation and its application to improve
	deep learning classifiers.
	\newblock {\em Applied Mathematics and Computation}, 389, 2021.
	
	\bibitem{PeirisSukhorukova}
	V.~Peiris and N.~Sukhorukova.
	\newblock The extension of the linear inequality method for generalized
	rational chebyshev approximation to approximation by general quasilinear
	functions.
	\newblock {\em Optimization}, 0(0):1--21, 2021.
	
	\bibitem{peiris2021}
	V.~Peiris, N.~Sukhorukova, and V.~Roshchina.
	\newblock Deep learning with nonsmooth objectives, 2021.
	
	\bibitem{raissi2018deep}
	M.~Raissi.
	\newblock Deep hidden physics models: Deep learning of nonlinear partial
	differential equations.
	\newblock {\em The Journal of Machine Learning Research}, 19(1):932--955, 2018.
	
	\bibitem{Rubinov00}
	AM. Rubinov.
	\newblock {\em Abstract Convexity and Global Optimization}.
	\newblock Kluwer Academic Publishers, New York, 2000.
	
	\bibitem{RubinovSimsek}
	AM. Rubinov and B.~Simsek.
	\newblock Conjugate quasiconvex nonnegative functions.
	\newblock {\em Optimization}, 35(1):1--22, 1995.
	
	\bibitem{min2017deep}
	M.~Seonwoo, L.~Byunghan, and Y.~Sungroh.
	\newblock Deep learning in bioinformatics.
	\newblock {\em Briefings in bioinformatics}, 18(5):851--869, 2017.
	
	\bibitem{miles2014}
	I.~Steponavi{\v{c}}{\.e}, RJ. Hyndman, K.~Smith-Miles, and L.~Villanova.
	\newblock Efficient identification of the pareto optimal set.
	\newblock In {\em International Conference on Learning and Intelligent
		Optimization}, pages 341--352. Springer, 2014.
	
	\bibitem{telgarsky2017neural}
	M.~Telgarsky.
	\newblock Neural networks and rational functions.
	\newblock In {\em International Conference on Machine Learning}, pages
	3387--3393. PMLR, 2017.
	
	\bibitem{SonyAIrobot}
	D.~Vail and M.~Veloso.
	\newblock Learning from accelerometer data on a legged robot.
	\newblock {\em IFAC Proceedings Volumes}, 37(8):822--827, 2004.	
	
	\bibitem{crouzeix1986note}
	JP~Crouzeix, JA~Ferland, and S~Schaible.
	\newblock A note on an algorithm for generalized fractional programs.
	\newblock {\em Journal of Optimization Theory and Applications},
	50(1):183--187, 1986.
\end{thebibliography}
\end{document}